\documentclass[a4paper]{amsart}

\usepackage{amsmath}
\usepackage{amssymb}
\usepackage{amsthm}

\usepackage{hyperref}

\usepackage[T1]{fontenc}

\usepackage{tikz-cd}

\usepackage{tikz}
\usetikzlibrary{shapes,fit,positioning,calc,matrix}
\tikzset{
  optree/.style={scale=.5,thick,grow'=up,level distance=10mm,inner sep=1pt},
  comp/.style={draw=none,circle,fill,line width=0,inner sep=0pt},
  dot/.style={draw,circle,fill,inner sep=0pt,minimum width=3pt},
  circ/.style={draw,circle,inner sep=1pt,minimum width=4mm},
  emptycirc/.style={draw,circle,inner sep=1pt,minimum width=2mm},
  root/.style={level distance=10mm,inner sep=1pt},
  leaf/.style={draw=none,circle,fill,line width=0,inner sep=0pt},
  nodot/.style={draw,circle,inner sep=1pt},
}

\theoremstyle{plain}
\newtheorem{theorem}{Theorem}
\numberwithin{theorem}{section}
\newtheorem{lemma}[theorem]{Lemma}
\newtheorem{proposition}[theorem]{Proposition}
\newtheorem{corollary}[theorem]{Corollary}

\theoremstyle{remark}
\newtheorem*{remark*}{Remark}
\newtheorem{remark}[theorem]{Remark}

\theoremstyle{definition}
\newtheorem{definition}[theorem]{Definition}
\newtheorem{example}[theorem]{Example}

\newtheorem{problem}{Problem}
\newtheorem*{problem*}{Problem}

\DeclareMathOperator{\Tw}{Tw}

\newcommand{\lmod}{{\operatorname{-Mod}}}

\newcommand{\cA}{\mathcal{A}}
\newcommand{\cB}{\mathcal{B}}
\newcommand{\cC}{\mathcal{C}}
\newcommand{\cD}{\mathcal{D}}

\begin{document}

\title{Operadic categories and quasi-Gr\"obner categories}
\author{Sergei Burkin}
\date{\today}
\address{Rua Marquês de São Vicente, 225, Edifício Cardeal Leme, sala 862, Gávea - Rio de Janeiro, CEP: 22451-900, Brazil}
\email{sburkin@protonmail.com}

\begin{abstract}
    Quasi-Gr\"obner categories were introduced by Sam and Snowden to unify treatment of categories in representation stability. We give new examples of quasi-Gr\"obner categories. Most of these categories are operadic categories of Batanin and Markl which are used to encode homotopy coherent structures. This suggests that other operadic categories might also be quasi-Gr\"obner. Additionally, we show that set-operads form a full subcategory of the category of operadic categories. We state several open problems.
\end{abstract}

\maketitle

\section*{Introduction}

For a category $\cC$ a $\cC$-module is a functor from $\cC$ to the category of $k$-modules for some ring $k$. Representation stability deals with naturally occurring modules over categories, and allows to give a good estimate on the dimension of these modules. To make this estimate one first  has to prove that for a given category $\cC$ and a sufficiently nice ring $k$ the category of $\cC$-modules over $k$ is locally Noetherian, i.e.\@ that any submodule of a finitely generated $\cC$-module over $k$ is finitely generated. To prove this one can use the notion of a \emph{quasi-Gr\"obner} category of Sam and Snowden (\cite{sam2017grobner}): for any quasi-Gr\"obner category $\cC$ and a left-Noetherian ring $k$ the category of $\cC$-modules over $k$ is locally Noetherian. Most categories of interest in representation stability are quasi-Gr\"obner.

Our main aim is to give new examples of quasi-Gr\"obner categories. These categories are related to operads: either they can be constructed from operads or at least they have the structure of the opposite of an operadic category of Batanin and Markl (\cite{Batanin2015operadic}). But first we clarify the connection between operads and operadic categories that is of independent interest: we show in Proposition~\ref{prp:operads-embed-into-opcats} that the former are the full subcategory of the latter via the functor $\cC$ from \cite{deBrito2018catp}. 

We expect that familiar properties of an operad $P$ can be used in the study of modules over the corresponding category $\cC(P)^{op}$. This is suggested by the example of \cite{tosteson2019representation}, which uses Koszulity of $Com$ in the study of modules over the corresponding PROP. Categories $\cC(P)^{op}$ are closely related to PROPs: for any operad $P$ there is a discrete opfibration $\cC(P)^{op}\to PROP(P)$. This motivates us to study how discrete opfibrations relate quasi-Gr\"obner property of their source and target.

Most of the new examples of quasi-Gr\"obner categories are related to operads in yet another way: these categories encode structures of (generalized) operads. Examples include the category $\mathcal{C}(pOp)^{op}$, the subcategory of active morphisms of the planar version of dendroidal category $\Omega$ of Moerdijk--Weiss \cite{moerdijk2007dendroidal}; and the category $\mathcal{C}(mOp_{(g,n)})^{op}$, which, after a minor modification, is essentially the subcategory of active morphisms of the category $U$ of Hackney--Robertson--Yau (\cite{hackney2020graphical}), and is also related to the operadic category $\mathtt{Gr}$ from \cite[Section~3]{batanin2018operadic}. The objects of these categories are trees and graphs with half-edges endowed with additional structure, and morphisms are subgraph insertions. We prove in Theorem~\ref{thm:pop-grobner} and Theorem~\ref{thm:mopgn-grobner} that these two categories are quasi-Gr\"obner. The proofs are a bit involved, and rely on the main lemma from \cite{barter2015noetherianity} used in the proof of the quasi-Gr\"obner property of the category $\mathbf{PT}$, which is also a category of trees, but not directly related to operads. If we restrict ourselves to the full subcategories on trees and graphs without vertices of arity 0 (i.e.\@ without vertices of degree 1), then the proofs of the two statements become much simpler. 

These categories are fairly involved and a priori there is no reason why these categories would be quasi-Gr\"obner. This leads to the following open problem.

\begin{problem} 
Is there any reasonable property of an operad $P$ that implies that the category $\cC(P)^{op}$ is quasi-Gr\"obner?
\end{problem}

The author expected that quasi-Gr\"obner property of categories might be related to the theory of Gr\"obner bases for operads, since the operads $pOp$ and $mOp_{(g,n)}$ have particularly nice Gr\"obner bases. Yet it seems that there is no such connection. 

There are several functors related to the functor $\cC$. Some of these functors produce classical quasi-Gr\"obner categories from the operads $uAs$ and $uCom$ of monoids and commutative monoids, and in some cases we get new examples related to non-commutative sets and cyclic sets. However, Proposition~\ref{pr:dendroidal-not-grobner} demonstrates that in general we cannot expect the categories $\Tw(P)$ and $\mathcal{U}(P)$ to be quasi-Gr\"obner even for nice operads $P$. This observation itself might be useful, since it suggests that one might want to focus on subcategories $\mathcal{C}(P)^{op}$ and $PROP(P)$ of $\Tw(P)$ and $\mathcal{U}(P)$, i.e.\@ on their subcategories of active morphisms. This is why categories of active morphisms appear in the following problem about Joyal's categories $\Theta_n$, which are known to be related to little disks operads~$E_n$.

\begin{problem}
     Are Joyal's categories $\Theta_n$ or their subcategories of active morphisms quasi-Gr\"obner? Are the categories of modules over these categories locally Noetherian?
\end{problem}

Another new example of a quasi-Gr\"obner category we give is the category $\mathbf{ncCS}$ that has connected surfaces with boundary as objects and cobordisms without caps as morphisms. We expect that the category $\mathbf{ncCS}$ might serve instead of $\mathbf{FS}^{op}$ as the indexing category of modules related to modular operads (see \cite{tosteson2019representation} for an example of such a module).

\subsection*{Structure of the paper.}
In the first part of Section~\ref{sec:lemmas} we recall basic notions in representation stability. Since categories in this work usually appear in pairs connected by discrete opfibrations, in the second part we explain how discrete opfibrations relate Gr\"obner and quasi-Gr\"obner properties of their \'etale and base categories. In Section~\ref{sec:operadic-categories} we recall the construction $\mathcal{C}$ of categories from operads and observe that it produces operadic categories. In Section~\ref{sec:graphs} we give ad hoc proofs that categories $\mathcal{C}(P)^{op}$ for several operads $P$ that encode generalized operads are quasi-Gr\"obner. In Section~\ref{sec:general-tw} we describe another construction of categories from operads and show that a category related to 2-cobordisms and modular operads is quasi-Gr\"obner. In Section~\ref{sec:further-examples} we give further examples of quasi-Gr\"obner categories related to operads, classical and new. 

\subsection*{Prerequirements.} We assume that the reader is familiar with operads. We only consider (coloured) operads in the category of sets. 

\subsection*{Acknowledgements.} The author is grateful to Anton Khoroshkin for his remarks and suggestions. The author also thanks the University of Tokyo and Pontifical Catholic University of Rio de Janeiro for their hospitality. 

\section{Preliminaries}
\label{sec:lemmas}

We recall basic facts about quasi-Gr\"obner categories introduced in \cite{sam2017grobner} and prove a few lemmas used in the main part.

\subsection{Quasi-Gr\"obner categories}

Let $\mathcal{C}$ be a category and $k$ be a ring.

\begin{definition}
    A $\mathcal{C}$-module $M$ over $k$ is a functor $M$ from $\mathcal{C}$ to the category $k\lmod$ of left $k$-modules. A morphism of $\mathcal{C}$-modules is a natural transformation of functors. An element of $M$ is an element of $M(c)$ for some object $c$ of $\mathcal{C}$. A submodule of $M$ generated by a subset $S$ of elements of $M$ is the smallest submodule of $M$ containing~$S$.
\end{definition}

\begin{definition}
    A $\mathcal{C}$-module $M$ over $k$ is Noetherian if every submodule of $M$ is finitely generated. The category of $\mathcal{C}$-modules over $k$ is locally Noetherian if every finitely generated $\mathcal{C}$-module is Noetherian.
\end{definition}

\begin{definition}
    Let $\cC$ be a small category and $c$ be an object in $\cC$. An admissible order on the set of morphisms from $c$ is a lift of the functor $Hom(c,-):\cC\to Set$ to a functor $\cC\to\mathbf{WO}$, where $\mathbf{WO}$ is the category of well-orders and strictly order preserving maps. 
\end{definition}

\begin{remark} An admissible order on the set of morphisms from $c$ endows for each object $c'$ of $\mathcal{C}$ the set $Hom(c,c')$ with a well-order $\preceq_{c'}$ such that for any $f,f':c\to c'$ and any $g:c'\to c''$ if $f\prec_{c'} f'$ then $g\circ f\prec_{c''} g\circ f'$. In the opposite direction, an admissible order on the set of morphisms from $c$ can be constructed from such a collection of compatible well-orders on the sets $Hom(c,c')$ for all $c'$.
\end{remark}

\begin{definition}
    Let $\cC$ be a small category. For any object $c$ in $\cC$ the set of morphisms from $c$ in $\cC$ has a canonical preorder $\leq$, with $f\leq g$ if there is $h$ such that $h\circ f=g$. Recall that any preorder $X$ generates a poset $X/{\sim}$ by identifying all the elements $a$ and $b$ in $X$ such that $a\leq b$ and $b\leq a$. The poset $|c/\cC|$ is the poset generated by the preorder $\leq$ on the set of morphisms from $c$ in $\cC$. 
\end{definition}

\begin{definition}
    A poset $X$ is Noetherian if for any sequence $x_1,x_2,\dots$ in $X$ there is $i < j$ such that $x_i\leq x_j$. Equivalently, $X$ is Noetherian if it satisfies descending chain condition and does not admit infinite anti-chains.
\end{definition}

\begin{definition}
    A category $\mathcal{C}$ is Gr\"obner if for all objects $c$ in $\cC$:
    \begin{itemize}
        \item[(G1)] the set of morphisms from $c$ can be endowed with an admissible order, and
        \item[(G2)] the poset $|c/\cC|$ is Noetherian.
    \end{itemize}
\end{definition}

\begin{definition}
    A functor $\Phi:\cC\to\cD$ satisfies \textbf{property (F)} if for any object $d$ of $\cD$ there exist finitely many objects $c_1,\dots,c_n$ in $\mathcal{C}$ and morphisms $f_i:d\to\Phi(c_i)$ such that for any object $c$ in $\cC$ any morphism $f:d\to \Phi(c)$ can be factored as $\Phi(g)\circ f_i$ for some $i$ and some $g:c_i\to c$ in $\mathcal{C}$.
\end{definition}

\begin{definition}
    A category $\cD$ is quasi-Gr\"obner if it admits an essentially surjective functor $\Phi:\cC\to\cD$ that satisfies property (F), with $\cC$ Gr\"obner. 
\end{definition}

\begin{remark} 
A composition of functors that satisfy property (F) satisfies property (F). Thus if $\cC$ is quasi-Gr\"obner and $\Phi:\cC\to\cD$ is essentially surjective and satisfies property (F), then $\cD$ is quasi-Gr\"obner.
\end{remark}

The following was shown in \cite[Theorem~4.3.2]{sam2017grobner}.

\begin{theorem}
Let $\cC$ be a quasi-Gr\"obner category and $k$ be a left-Noetherian ring. Then the category of $\cC$-modules over $k$ is locally Noetherian.
\end{theorem}

\begin{definition}
    A functor $\Phi:\mathcal{C}\to\mathcal{D}$ satisfies \textbf{property (S)} if for any morphisms $f:c\to c'$ and $g:c\to c''$ in $\cC$ such that there is $h':\Phi{c'}\to \Phi{c''}$ in $\cD$ with $\Phi(g)=h'\circ\Phi(f)$, there is a morphism $h$ in $\cC$ such that $g=h\circ f$.
\end{definition}

\begin{proposition}[Proposition 4.4.2 in \cite{sam2017grobner}]
    If $\Phi:\cC\to\cD$ is faithful and satisfies property (S) and $\cD$ is Gr\"obner, then $\cC$ is Gr\"obner.
\end{proposition}

\subsection{Lemmas}

Most of the categories that we consider come in pairs connected by discrete opfibrations. This motivates the following lemmas.

\begin{definition}
    A functor $G:\cC\to\cD$ is a discrete opfibration if for any object $c$ in $\cC$ and any morphism $g:G(c)\to d$ in $\cD$ there is unique morphism $f$ from $c$ in $\cC$ such that $G(f)=g$. The morphism $f$ is called the lift of $g$ to $c$.
\end{definition}

Notice that essentially surjective discrete opfibrations are surjective on objects.

\begin{lemma}
\label{lem:grobner-opfibration-left-right}
Let $G:\cC\to\cD$ be a discrete opfibration surjective on objects. If $\cC$ is quasi-Gr\"obner, then $\cD$ is quasi-Gr\"obner.
\end{lemma}
\begin{proof}
Discrete opfibrations that are surjective on objects satisfy property (F), with the identity morphisms $id_d$ being the morphisms $f_1$ that ensure property (F), with $c_1$ being any preimage of $d$.
\end{proof}

\begin{lemma}
\label{lem:faithful-g1}
Let $\mathcal{D}$ be a category that satisfies (G1) and let $G:\mathcal{C}\to\mathcal{D}$ be a faithful functor. Then $\mathcal{C}$ satisfies (G1).
\end{lemma}
\begin{proof}
Let $c$ be an object of $\mathcal{C}$ and $\prec$ be an admissible order on the morphisms from $G(c)$ in $\mathcal{D}$ that exists by the property (G1). For every morphism $h$ from $G(c)$ in $\mathcal{D}$ choose any well-order $\prec_h$ on the set $G^{-1}(c,h)$ of morphisms $f$ from $c$ such that $G(f)=h$.  Define the order $\prec'$ on the set of morphisms from $c$ in $\cC$ so that $f\prec' f'$ if $G(f)\prec G(f')$ or if $G(f)=G(f')$ and $f\prec_{G(f)} f'$. This is a well-order. 

To show that $\prec'$ is admissible, let $f,f':c\to c'$ and $g:c'\to c''$ be morphisms in $\cC$ with $f\prec' f'$. By faithfulness $G(f)\neq G(f')$, and thus $G(f)\prec G(f')$. By admissibility of $\prec$ we have $G(g\circ f)\prec G(g\circ f')$, and thus $g\circ f\prec' g\circ f'$. The order $\prec'$ is admissible.
\end{proof}

\begin{lemma}
\label{lem:discrete-opfibrations-grobner}
Let $G:\mathcal{C}\to\mathcal{D}$ be a discrete opfibration. If $\mathcal{D}$ is Gr\"obner, then $\mathcal{C}$ is Gr\"obner.
\end{lemma}
\begin{proof}
Any slice category $c/\cC$ is isomorphic to the corresponding slice category $G(c)/\cD$. Thus if $\cD$ satisfies property (G2) then $\cC$ satisfies property (G2). Discrete opfibrations are faithful, and property (G1) follows from Lemma~\ref{lem:faithful-g1}.
\end{proof}

\begin{lemma}\label{lem:opfibration-quasi-grobner}
Let $G:\mathcal{A}\to\mathcal{B}$ be a discrete opfibration. If $\mathcal{B}$ is quasi-Gr\"obner, then $\mathcal{A}$ is quasi-Gr\"obner.
\end{lemma}
\[\begin{tikzcd}
	\cC & \cD \\
	\cA & \cB
	\arrow[from=1-1, to=1-2]
	\arrow["\Phi", from=1-2, to=2-2]
	\arrow["{\Phi'}"', from=1-1, to=2-1]
	\arrow["G"', from=2-1, to=2-2]
\end{tikzcd}\]
\begin{proof}
Since $\cB$ is quasi-Gr\"obner, there is essentially surjective functor $\Phi:\cD\to \cB$ that satisfies property (F), with $\cD$ Gr\"obner. Let $\cC$ be the pullback of $\Phi:\cD\to \cB$ and $G:\cA\to \cB$, and $\Phi':\cC\to \cA$ be the functor in the pullback square. Recall that pullbacks of categories are formed by taking fibered products of sets: the objects (and the morphisms) in $\cC$ are the pairs of objects (respectively the pairs of morphisms) in $\cA$ and $\cD$ that have the same image in $\cB$, and the functors $\cC\to \cA$ and $\cC\to \cD$ are the projection maps.

Discrete opfibrations are stable under pullbacks, thus $\cC\to \cD$ is a discrete opfibration and by Lemma~\ref{lem:discrete-opfibrations-grobner} the category $\cC$ is Gr\"obner. 

To show that $\Phi':\cC\to \cA$ is essentially surjective, let $A$ be an object in $\cA$. Since $\Phi:\cD\to \cB$ is essentially surjective, there is an object $D$ in $\cD$ such that there is an isomorphism $G(A)\to\Phi(D)$ in $\cB$. Let $A\to A'$ be the lift of this isomorphism to $\mathcal{A}$. Lifts of isomorphisms are isomorphisms, thus $A'$ is isomorphic to $A$. The pair $(A',D)$ is an object of $\mathcal{C}$, and its image under $\Psi'$ is isomorphic to $A$. Thus $\Phi'$ is essentially surjective.

Next we prove the property (F) for $\Phi'$. Let $A$ be an object in $\cA$. The property (F) of $\Phi$ with respect to $G(A)$ gives  objects $D_i$ in $\cD$ and morphisms $f_i:G(A)\to \Phi(D_i)$ in $\cB$. Let $f'_i:A\to A_i$ be the lifts of $f_i$ to $\cA$. The pairs $(A_i, D_i)$ are objects of $\cC$. We will show that $f'_i$ and $(A_i, D_i)$ ensure the property (F) of $\Phi'$ with respect to $A$. Let $(A', D')$ be an object of $\cC$, and $f':A\to A'$ be a morphism in $\cA$. We have to find an index $i$ and a morphism $g':(A_i,D_i)\to (A',D')$ in $\cC$ such that $f'=\Phi'(g')\circ f'_i$. Let $g:D_i\to D'$ be the morphism in $\cD$ given by the property (F) of $\Phi$ with respect to $G(f'):G(A)\to G(A')$ and $D'$ (this is possible to do since $G(A')=\Phi(D')$), i.e.\@ such that $G(f')=\Phi(g)\circ f_i$. Let $g':(A_i, D_i)\to (X,D')$ be the lift of $g$ to $\cC$. We have that $G(\Phi'(g')\circ f'_i)=\Phi(g)\circ f_i=G(f')$. Opfibration property of $G$ implies that $\Phi'(g')\circ f'_i=f'$, as was desired.
\end{proof}

\section{The main construction and operadic categories}
\label{sec:operadic-categories}

We show that the construction from \cite{deBrito2018catp} of categories from operads produces operadic categories of Batanin and Markl \cite{Batanin2015operadic}. The opposites of these categories will be the main examples of quasi-Gr\"obner categories in this work.

\begin{definition}
    The category $\mathcal{S}$ is the full subcategory of the category of finite sets on objects $\underline{n}=\{0,\dots,n\}$ for all $n\geq 0$.
\end{definition}

\begin{definition}[{\cite[Example 7.1]{deBrito2018catp}}]
Let $P$ be an operad in the category of sets. Define the category $\cC(P)$ as follows. The objects of $\cC(P)$ are the operations of $P$. A morphism $f:q\to p$ in $\cC(P)$ is a \emph{2-level tree}, i.e.\@ a planar tree (with leaves) of height $2$ with vertices marked by operations of $P$ so that:
\begin{itemize}
    \item the number of input edges of a vertex is equal to the arity of the operation marking this vertex,
    \item the lowest vertex is marked by $p$, the target of $f$, and the upper vertices are marked by some operations $q_i$,
    \item the input leaves are adjacent to the upper vertices, and not to the lowest vertex,
    \item the input leaves are indexed via permutation of $\underline{n}$, where $n$ is the arity of $q$, so that for each upper vertex $q_i$ the indices of the leaves above it increase in planar order,
    \item $q$ is equal to the operadic composition of $p$ with the operations $q_i$, further permuted via the permutation on leaves.
\end{itemize}
The composition is computed by grafting the upper vertices of one 2-level tree into the leaves of the other according to the permutation on leaves, and contracting all the subtrees in the obtained tree of height $3$ that are above the lowest vertex, see Figure~\ref{fig:2leveltrees}. Associativity of composition in these categories follows from associativity of operadic composition. This construction gives a functor $\cC$ from set-operads to categories.
\end{definition}

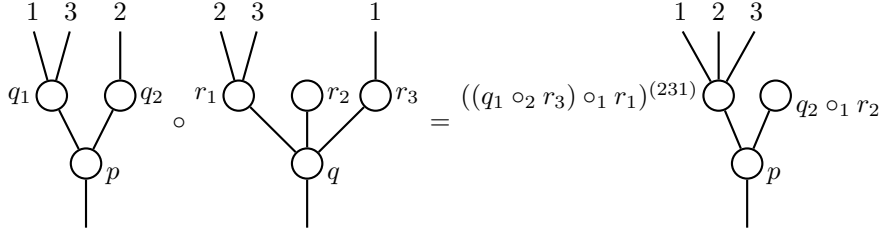
\begin{figure}[t]
\begin{align*}
  \begin{aligned}\begin{tikzpicture}[optree,
    level distance=18mm,
    level 2/.style={sibling distance=18mm},
    level 3/.style={sibling distance=10mm}]
  \node{}
  child { node[circ,label=-3:$p$]{}
      child { node[circ,label=180:$q_1$]{}
        child { node[label=above:$1$]{} }
        child { node[label=above:$3$]{} } }
      child { node[circ,label=0:$q_2$]{}
        child { node[label=above:$2$]{} } } };
\end{tikzpicture}\end{aligned}
\circ
\begin{aligned}\begin{tikzpicture}[optree,
    level distance=18mm,
    level 2/.style={sibling distance=18mm},
    level 3/.style={sibling distance=10mm}]
  \node{}
  child { node[circ,label=-3:$q$]{}
      child { node[circ,label=180:$r_1$]{}
        child { node[label=above:$2$]{} }
        child { node[label=above:$3$]{} } }
      child { node[circ,label=0:$r_2$]{} }
      child { node[circ,label=0:$r_3$]{}
        child { node[label=above:$1$]{} } } };
\end{tikzpicture}\end{aligned}
=
\begin{aligned}\begin{tikzpicture}[optree,
    level distance=18mm,
    level 2/.style={sibling distance=15mm},
    level 3/.style={sibling distance=10mm}]
  \node{}
  child { node[circ,label=-3:$p$]{}
      child { node[circ,label=180:$((q_1\circ_2 r_3)\circ_1 r_1)^{(231)}$]{}
        child { node[label=above:$1$]{} }
        child { node[label=above:$2$]{} }
        child { node[label=above:$3$]{} }}
      child { node[circ,label=-3:$q_2\circ_1 r_2$]{} } };
\end{tikzpicture}\end{aligned}
\end{align*}
\caption{A composition of two morphisms in $\mathcal{C}(P)$. Notice that $q$ is equal to $((p\circ_2 q_2)\circ_1 q_1)^{(132)}$.}
\label{fig:2leveltrees}
\end{figure}

\begin{example}\label{ex:c-ucom}
    For the terminal operad $uCom$ of commutative monoids the category $\cC(uCom)$ is isomorphic to the category $\mathcal{S}$. Since $\mathcal{C}$ is a functor, any category $\cC(P)$ is endowed with canonical functor to $\mathcal{S}$.
\end{example}

\begin{proposition}
\label{prp:operads-embed-into-opcats}
    For any operad $P$ (in the category of sets) the category $\mathcal{C}(P)$ is canonically endowed with a structure of an operadic category. The functor $\mathcal{C}:Operads\to OpCats$ is a fully faithful functor from set-operads to operadic categories and operadic functors. 
\end{proposition}
\begin{proof}
Operadic categories were introduced in \cite{Batanin2015operadic}. Here we use the definition and notation from \cite{garner2021operadic}. 

Let $P$ be an operad. The local terminal objects in $\mathcal{C}(P)$ are defined to be the operations $id_c$ for all colours $c$ of $P$: for any operation $p$ in $P$ there is exactly one morphism in $\mathcal{C}(P)$ from $p$ to some $id_c$, namely to $id_{c_0}$, where $c_0$ is the output colour of $p$.

The cardinality functor $|-|:\mathcal{C}(P)\to \mathcal{S}$ is the functor from Example~\ref{ex:c-ucom}. It sends an operation $p$ of arity $n$ to $\underline{n}$, and sends a 2-level tree $f:q\to p$ to the set map $|f|$ such that for all $i$ the preimage of $i$ consists of all the indices of leaves above the $i$-th upper vertex of $f$.

The fibre functors are defined as follows. On objects the functor $\varphi_{p,i}$ sends a 2-level tree $f:q\to p$ to the operation $q_i$ that marks the $i$-th upper vertex of $f$. On morphisms $\varphi_{p,i}$ sends a morphism $g:r\to q$ over $f:q\to p$ to the 2-level tree $g^f_i$ with the lowest vertex marked by $q_i$; with upper vertices marked consecutively by the operations $r_j$ that mark the vertices of $g$ for all $j$ that are the indices of leaves above the $i$-th vertex of $f$; with indices above the leaves of these vertices $r_j$ the same as in $g$, except these indices are shifted via the order preserving bijection with $|gf^{-1}(i)|$. In other words, the 2-level tree $g^f_i$, up to the shift of indices, is subtree of the tree of height $3$ that corresponds to the composition $g\circ f$, this subtree containing the vertex marked by $q_i$ and all the vertices above it. 

The verification that the map $\varphi_{p,i}$, of the axioms of operadic category and of the fully faithfulness of the functor is left to the reader, as it is straightforward. Notice that the action of the functors $\varphi_{p,i}$ on the indices of leaves is the same as in the category $\mathcal{S}$. And since $\mathcal{S}$ is an operadic category, one does not have to check that the maps on indices of leaves satisfy the axioms. To check the second half of the axiom (A5) one might want to draw the tree of total height $4$ that corresponds to a composition $hgf$. For the fully faithfulness notice that any morphism in $\cC(P)$ is determined by its fibers. 
\end{proof}

\section{Graph-based categories}
\label{sec:graphs}

Next we consider the opposites of the categories $\mathcal{C}(P)$ for the operads $P$ that encode planar, symmetric, cyclic, modular and genus-graded modular operads. Description of these operads can be found in \cite[Section 1.1]{burkin2022twisted}. Operations of these operads, and thus objects in the corresponding categories $\mathcal{C}(P)$, are graphs with half-edges endowed with additional structure. To avoid confusion, the operations of these operads whose underlying graphs are trees will be called \emph{operadic trees}, while the trees that correspond to morphisms in categories $\mathcal{C}(P)$ will be called \emph{2-level trees}.

\begin{definition}
    Let $P$ be an operad. Its suboperad $nuP$ consists of all non-unital operations of $P$, i.e.\@ of operations of non-zero arity.
\end{definition}

To be precise, we are interested in the subcategories $\mathcal{C}(nuP)^{op}$ of categories $\mathcal{C}(P)^{op}$. This is justified by Lemma~\ref{lem:r-zero-to-full} that shows that if the number of operations of arity $0$ in $P$ is finite and $\mathcal{C}(nuP)^{op}$ is quasi-Gr\"obner then $\mathcal{C}(P)^{op}$ is quasi-Gr\"obner. 

There is a sequence of inclusions of operads $pOp \hookrightarrow sOp \hookrightarrow cOp \hookrightarrow mOp_{(g,n)}$. This gives a sequence $\mathcal{C}(nupOp)\hookrightarrow \mathcal{C}(nusOp)\hookrightarrow\mathcal{C}(nucOp)\hookrightarrow\mathcal{C}(numOp_{(g,n)})$ of inclusions of categories. Next we describe the category $\mathcal{C}(numOp_{(g,n)})^{op}$.

\begin{definition}
A graph with half edges is a finite set $V$ of vertices, a finite set $H$ of half-edges, an involution $inv$ on $H$, and the adjacency map $t:H\to V$. A fixed point of the involution is called a leaf. A two-element orbit of the involution is called an edge. The set $V$ together with the set of edges can be seen as the usual graph, where an edge $\{h_1, h_2\}$ connects the vertices $t(h_1)$ and $t(h_2)$.
\end{definition}

\begin{definition}
An operadic graph is a graph with half edges $(V,H,inv,t)$ together with an order on the leaves, for each vertex $v$ an order on the set $t^{-1}(v)$, and an order on vertices. The orders on half-edges and on leaves will often be given by bijection with sets $\{0,\dots,n\}$, while the order on vertices will be given by bijection with sets $\{1,\dots,n\}$.
\end{definition}

\begin{definition}
    A corolla is an operadic graph that is a tree with one vertex and with any permutation on leaves.
\end{definition}

\begin{proposition}
The category $\mathcal{C}(numOp_{(g,n)})^{op}$ has the following concrete description. Its objects are operadic graphs endowed additionally with a genus map $g:V\to\mathbb{N}$ from the vertices of the graph. A morphism $f:p\to q$, which is given by a 2-level tree with the root vertex marked by $p$ and the upper vertices marked by some operations $q_l$, corresponds to embedding of operadic graphs $q_l$ into the vertices of $p$. In other words the operadic graph $q$ is the union of operadic graphs $q_l$ according to $f$, in the following way. For all $l$ the number of leaves of the operadic graph $q_l$ coincides with the number of half-edges adjacent to the $l$-th vertex of $p$, and the genus of the $l$-th vertex of $p$ is equal to $(\sum_{v\in q_l} g(v) + g(q_l))$, where $g(v)$ is the genus of the vertex $v$ of $q_l$ and $g(q_l)$ is the genus of the graph of $q_l$. If the $i$-th half-edge of the $j$-th vertex in $p$ is connected to the $k$-th half-edge of the $l$-th vertex in $p$, then $f$ connects the $i$-th leaf of $q_j$ to the $k$-th leaf of $q_l$. If the $i$-th half-edge of the $j$-th vertex in $p$ is the $k$-th leaf of $p$, then the $i$-th leaf of $q_j$ becomes the $k$-th leaf of $q$. Additionally $f$ determines the order on vertices of $q$ according to the indices of leaves of the 2-level tree of $f$. The isomorphisms in $\mathcal{C}(numOp_{(g,n)})^{op}$ are precisely the morphisms such that the inserted graphs $q_j$ are corollas.
\end{proposition}

\begin{remark}
Instead of categories $\mathcal{C}(P)^{op}$ one may want to consider the wide subcategories $Active(P)$ of active morphisms of categories $\Tw(P)$ described in the next section. Here an active morphism is a morphism such that the lowest vertex in the corresponding 3-level tree has arity 1 and is marked by an invertible operation. Notice though that if an operad $P$ is such that for any colour $c$ of $P$ the number of invertible operations in $P$ with input colour $c$ is finite, then the inclusion of $\mathcal{C}(P)^{op}$ into $Active(P)$ is essentially surjective and satisfies property (F), and thus if $\mathcal{C}(P)^{op}$ is quasi-Gr\"obner then $Active(P)$ is quasi-Gr\"obner. This finiteness condition always holds for the operads that we consider.
\end{remark}

\begin{remark}
For any operad $P$ the image of the cardinality functor $|-|:\mathcal{C}(nuP)\to\mathcal{C}(uCom)\cong\mathcal{S}$ lies in the subcategory $\mathbf{FS}$ of finite surjections. The functor $|-|^{op}:\mathcal{C}(numOp_{(g,n)})^{op}\to\mathbf{FS}^{op}$ sends an operadic graph with $n$ vertices to $\{1,\dots,n\}$ and a morphism $f:p\to q$ that inserts $q_l$ into the $l$-th vertex of $p$ to the opposite of the map $h$ such that the preimage $h^{-1}(l)$ is the set of indices of vertices of the subgraph $q_l$ of $q$.
\end{remark}

Let $pOp$ be the operad of planar operads. The category $\mathcal{C}(pOp)^{op}$ is equivalent to the subcategory of active morphisms of the category $\Omega_{pl}$, the planar version of the Moerdijk--Weiss dendroidal category $\Omega$. The objects of $\mathcal{C}(nupOp)^{op}$ are the operadic graphs that are planar rooted trees with half-edges, with vertices indexed from $1$ to $n$, with leaves indexed in planar order. A morphism $f:p\to q$ in $\mathcal{C}(nupOp)^{op}$ embeds an operadic tree $q_j$ into the $j$-th vertex $v_j$ of the operadic tree $p$ for all $j$, and indexes the vertices of $q$ so that $v\prec v'$ in $q$ if $v$ and $v'$ belong to the same subgraph $q_j$ and $v\prec v'$ as the vertices of $q_j$, or if $v\in q_i$, $v'\in q_j$ and  $v_i\prec v_j$ as the vertices of $p$. 

\begin{theorem}
\label{thm:pop-grobner}
The category $\mathcal{C}(nupOp)^{op}$ is quasi-Gr\"obner.
\end{theorem}
\begin{proof}
The cardinality functor $|-|:\mathcal{C}(nupOp)^{op}\to\mathbf{FS}^{op}$ is faithful. Indeed, a morphism $h:p\to q$ in $\mathcal{C}(nupOp)^{op}$ substitutes certain operadic trees $q_j$ into the vertices of $p$ and assigns order to the vertices of $q$. Such a morphism $h$ is determined by the operations $p$ and $q$ and by the partition of the planar tree of $q$ into subtrees $q_j$, and this partition is determined by the indices of the vertices of the subtrees $q_j$ of $q$, i.e.\@ it is determined by $|h|$.

Let $\mathcal{D}$ be the full subcategory of $\mathcal{C}(nupOp)^{op}$ on operadic trees with vertices ordered in the clockwise depth-first search order starting from the root. The inclusion $\mathcal{D}\to\mathcal{C}(nupOp)^{op}$, an equivalence of categories, is essentially surjective and satisfies property (F). Observe that the image of the restriction of $|-|:\mathcal{C}(nupOp)^{op}\to\mathbf{FS}^{op}$ to $\mathcal{D}$ lies in $\mathbf{OS}^{op}$. Since $|-|:\mathcal{D}\to\mathbf{OS}^{op}$ is faithful and $\mathbf{OS}^{op}$ satisfies property (G1), by Lemma~\ref{lem:faithful-g1} the category $\mathcal{D}$ satisfies property (G1).

It remains to prove that $\mathcal{D}$ satisfies property (G2), and this is done in the lemma below. The main difficulty is posed by the operadic trees that contain vertices of arity $0$ (i.e.\@ vertices without input edges, equivalently vertices of total degree~$1$). If one is interested only in the full subcategory $\mathcal{D}_+$ of $\mathcal{C}(nupOp)^{op}$ on trees without vertices arity $0$, then one may proceed as follows. Let $f_i$ be a sequence of morphisms in $\mathcal{D}_+$ from the same object $p$. Let $q_{ji}$ be the operadic tree that is substituted into the $j$-th vertex of $p$ under the morphism $f_i$. For any fixed $j$ the trees $q_{ji}$ have the same number of input leaves, and we can choose a subsequence of $f_i$ such that the trees $q_{ji}$ are all homeomorphic to each other and differ only by the number of vertices of degree 2. We can further choose a subsequence such that the corresponding numbers of vertices of degree 2 between any two adjacent vertices of degree not equal to 2, and also the numbers of vertices of degree 2 between leaves and vertices of degree not equal to 2, are non-decreasing. Doing this for each $j$, we get a non-decreasing sequence of morphisms $f_i$.
\end{proof}

\begin{lemma}
The property (G2) for $\mathcal{D}$ is equivalent to the relative Kruskal's tree theorem proved in \cite{barter2015noetherianity}, i.e.\@ to the property (G2) for the category $\mathbf{PT}$ of that work. 
\end{lemma}
\begin{proof}
There is a functor $G:\mathbf{PT}\to\mathcal{D}$ that sends a planar rooted tree $T$ to the same tree  (without input leaves and with the root leaf added to the root vertex) and sends a morphism $f:T\to T'$ to the morphism that for each vertex $v$ of $T$ embeds into $v$ the maximal subtree of $T'$ that contains only the vertices above or equal to $f(v)$ and, for all the vertices $w$ that are the children of $v$, does not contain the vertices $f(w)$, yet contains the half-edges directly below the vertices $f(w)$.

There is a faithful functor $J:\mathcal{D}\to\mathbf{PT}$: on objects it removes the root leaf and adds a vertex of arity $0$ to each input leaf, on morphisms it sends $f:p\to q$ to the map of planar trees that sends a vertex $v$ of $p$ to the lowest vertex in the subtree of $q$ that is embedded into $v$ by $f$, and sends the vertices above the leaves of $p$ to the corresponding vertices above the leaves of $q$. 

The composition $J\circ G$ is the identity. Since $J$ is faithful, $G$ is fully faithful. This allows to view $\mathbf{PT}$ as a full subcategory of $\mathcal{D}$. In particular, if $\mathcal{D}$ satisfies property (G2), then $\mathbf{PT}$ satisfies property (G2).

For any object $p$ in $\mathcal{D}$ let $J':p/\mathcal{D} \to J(p)/\mathbf{PT}$ be the functor induced by $J$. Since $J$ is faithful, $J'$ is faithful. Observe that $J'$ is full. Indeed, $J$ adds vertices to the leaves of operadic trees in $\mathcal{D}$, and any morphism in $J(p)/\mathbf{PT}$ between the objects in the image of $J'$ has to preserve such new vertices, which allows to recover its preimage in $p/\mathcal{D}$. The fully faithfulness of $J'$ implies that if $\mathbf{PT}$ satisfies (G2), then $\mathcal{D}$ satisfies (G2).
\end{proof}

\begin{corollary}
\label{thm:sop-grobner}
The category $\mathcal{C}(nusOp)^{op}$ is quasi-Gr\"obner.
\end{corollary}
\begin{proof}
Let $\mathcal{D}'$ be the subcategory of $\mathcal{C}(nusOp)^{op}$ with objects being operadic trees with any permutation on leaves and with permutation on vertices given by the clockwise depth-first search order starting from the root; and with morphisms being such that the upper vertices in the 2-level trees that represent these morphisms belong to $pOp$, i.e.\@ the corresponding operations $q_j$ are operadic trees with trivial permutation on leaves. Notice that any operation in $sOp$ can be represented as $c\circ_1 r$ where $c$ is a corolla and $r$ is an operation in $pOp$. This implies that the inclusion $\mathcal{D}'\to\mathcal{C}(nusOp)^{op}$ is essentially surjective and satisfies property (F): any morphism $p\to q$ in $\mathcal{C}(nusOp)^{op}$ is a composition of an isomorphism that permutes input edges of vertices of the operadic tree $p$ by inserting corollas (and permutes the order on vertices) and of a morphism from $\mathcal{D}'$. 

Let $\mathcal{D}$ be the Gr\"obner category from Theorem~\ref{thm:pop-grobner}. The functor $\mathcal{D}'\to \mathcal{D}$ that forgets the indices of leaves of trees is a discrete opfibration. By Lemma~\ref{lem:discrete-opfibrations-grobner} the category $\mathcal{D}'$ is Gr\"obner.
\end{proof}

\begin{proposition}
The inclusion $\mathcal{C}(nusOp)\to\mathcal{C}(nucOp)$ induced by the inclusion of operads $sOp\to cOp$ is an equivalence of categories.
\end{proposition}
\begin{proof}
Any object $p$ in $\mathcal{C}(nucOp)$ is isomorphic to an object of $\mathcal{C}(nusOp)$ via the morphism that substitutes  into the vertices of $p$  corollas with cyclically permuted leaves. For any object $p$ in $\mathcal{C}(nusOp)$ if operations $q_j$ from $nucOp$ are such that the substitution of these into $p$ gives an object in $\mathcal{C}(nusOp)$, then the operations $q_j$ are in $nusOp$. Thus $\cC(nusOp)$ is a full subcategory of $\cC(nucOp)$.
\end{proof}

The next two categories are closely related to the operadic category $\mathtt{Gr}$ from \cite[Section~3]{batanin2018operadic} and also to the opposite of the subcategory of active morphisms of the category $U$ from \cite{hackney2020graphical}.

\begin{proposition}
The category $\mathcal{C}(numOp)^{op}$ is not quasi-Gr\"obner.
\end{proposition}
\begin{proof}
The category $\mathcal{C}(numOp)^{op}$ is similar to the category $\mathcal{C}(numOp_{(g,n)})^{op}$ described above, except the objects are not endowed with the genus map and the morphisms are not required to respect the genus map.

Observe that there are morphisms  $f_i:id_1\to p_i$ from the corolla with two leaves, ordered trivially, to the graphs on two vertices, of genus $i$, without loops, and with the $0$-th leaf adjacent to the first vertex and the $1$-st leaf adjacent to the second vertex. Let $M$ be the module generated by $id_1$ and let $N$ be the maximal submodule of $M$ that is trivial on the graphs of genus $0$ and on the graphs with only one vertex. The module $N$ is non-trivial over objects $p_i$, and thus is not finitely generated.
\end{proof}

\begin{theorem}
\label{thm:mopgn-grobner}
The category $\mathcal{C}(numOp_{(g,n)})^{op}$ is quasi-Gr\"obner.
\end{theorem}
\begin{proof}
Denote $\mathcal{C}(numOp_{(g,n)})^{op}$ as $\mathcal{C}$. The proof proceeds in several steps, in the end giving a sequence $\mathcal{G}\to\mathcal{Z}\hookrightarrow\cC'\hookrightarrow\cC$ of categories such that $\mathcal{G}$ is quasi-Gr\"obner and the functors are sufficiently nice.

Notice that morphisms in $\mathcal{C}$ preserve the number of leaves of operadic graphs. Let $\mathcal{C}_0$, $\mathcal{C}_1$ and $\mathcal{C}'$ be the full subcategories of $\mathcal{C}$ on operadic graphs without leaves, with exactly one leaf, and with at least one leaf respectively. There is a functor $\mathcal{C}_1\to\mathcal{C}_0$ that removes the leaf. This functor is full and surjective on objects, and in particular satisfies property (F).  If $\mathcal{C}'$ is quasi-Gr\"obner, then $\mathcal{C}_1$ is quasi-Gr\"obner, and then $\mathcal{C}_0$ is quasi-Gr\"obner, which implies that $\mathcal{C}$, as the disjoint union of $\mathcal{C}'$ and $\mathcal{C}_0$, is quasi-Gr\"obner.

As in the previous proofs, to prove that $\mathcal{C}'$ is quasi-Gr\"obner we take a subcategory $\mathcal{Z}$ of $\mathcal{C}'$ such that $\mathcal{Z}$ does not have non-trivial endomorphisms and the image of $\mathcal{Z}$ under the functor $F:\mathcal{C}'\to\mathbf{FS}^{op}$ lies in $\mathbf{OS}^{op}$. The objects of the subcategory $\mathcal{Z}$ are operadic graphs with at least one leaf, with the first vertex being the vertex adjacent to the $0$-th leaf, and with the $0$-th half-edge of the first vertex being the $0$-th leaf of the graph, with vertices ordered in clockwise depth-first search order, and such that the edges that are traversed by the clockwise depth-first search that starts from the first vertex contain exactly one $0$-th half-edge and one non-zeroth half-edge of adjacent vertices. The morphisms $f:p\to q$ of $\mathcal{Z}$ are the morphisms of $\mathcal{C}'$ such that for all $l$ the order on the leaves of the operadic graph $q_l$ that is inserted by $f$ into the vertices of $p$ is the order in which the clockwise depth-first search over $q_l$ that starts from the $0$-th leaf traverses the leaves of $q_l$. This condition on the morphisms implies that the $0$-th leaf of $q_l$ is the $0$-th half-edge of its vertex and that the clockwise depth-first search over $q$ that starts from the first vertex traverses (the smallest vertices of) the subgraphs $q_l$ of $q$ sequentially from $q_1$ to $q_n$. The latter property implies that the image of $\mathcal{Z}$ under $F:\mathcal{C}'\to\mathbf{FS}^{op}$ lies in $\mathbf{OS}^{op}$. Endomorphisms in $\mathcal{C}'$ insert corollas into vertices, and the condition on the morphisms in $\mathcal{Z}$ implies that these corollas are identity operations, i.e.\@ endomorphisms in $\mathcal{Z}$ are trivial.

The inclusion $\mathcal{Z}\to\mathcal{C}'$ is essentially surjective: any object of $\mathcal{C}'$ can be obtained from an object of $\mathcal{Z}$ by insertion of corollas and by permutation of indices of vertices. Next we show that this inclusion satisfies property (F). For an object $d$ of $\mathcal{C}'$ the morphisms that ensure the property (F) will be all the isomorphisms from $d$ to objects of $\mathcal{Z}$. Notice that we can take some isomorphism $\xi:d\to d'$ with $d'$ in $\mathcal{Z}$, and then any isomorphism $d\to z$ with $z$ in $\mathcal{Z}$ is a composition of $\xi$  with some isomorphism $d'\to z$ in $\mathcal{C}'$. This shows that in the proof of the property (F) we can assume that $d$ itself is in $\mathcal{Z}$. 

Observe the following. Let $p$ be an object of $\mathcal{Z}$, and let $i:p\to r$ be an isomorphism in $\mathcal{C}'$ that for some $j$ substitutes into the $j$-th vertex of $p$ a corolla whose $0$-th leaf is the $0$-th half-edge of its vertex, substitutes the identity operations into the remaining vertices of $p$, and does not permute the indices of vertices. Then there is exactly one isomorphism $i':r\to p'$ in $\mathcal{C}'$ with $p'$ in $\mathcal{Z}$ that substitutes the identity operations into the first $j$ vertices of $r$, for all $l>j$ substitutes corollas with cyclic permutation on leaves into the remaining vertices of $r$, and does not permute the indices of the first $j$ vertices of $r$.

Let $f:p\to q$ be a morphism in $\mathcal{C}'$ between objects in $\mathcal{Z}$. Let $q_l$ be the operadic graph that is inserted into the $l$-th vertex of $p$ by $f$. Since the $0$-th leaf of $q$ is the $0$-th half-edge of its vertex, the $0$-th leaf of $q_1$ is the $0$-th half-edge of its vertex. The operation $q_1$ can be represented as the composition $x_1\circ_1 q'_1$, where $x_1$ is a corolla with the $0$-th leaf being the $0$-th half-edge of its vertex, and with $q'_1$ such that the order on the leaves is the order in which the clockwise depth-first search traverses the leaves of $q'_1$. The substitution of the operation $x_1$ into the first vertex of $p$ gives an isomorphism from $p$ in $\mathcal{C}'$. As explained in the previous paragraph, this isomorphism can be extended to isomorphism $i_1:p\to p_1$ with $p_1$ in $\mathcal{Z}$. This gives decomposition of $f$ as $f_1\circ i_1$. If $p$ has $n$ vertices, doing the same for $f_{i-1}$ and the $i$-th vertex of $p_{i-1}$, with $i$ ranging from $2$ to $n$, gives the isomorphism $i=i_{n-1}\circ\dots\circ i_1:p\to p_n$ such that $f=f_n\circ i$. The morphism $f_n$ substitutes operations $q'_l$ into vertices of $p_n$, and the orders on leaves of $q'_l$ are such that $f_n$ is in $\mathcal{Z}$. This shows that the inclusion $\mathcal{Z}\to \mathcal{C}'$ satisfies property (F). 

It remains to prove that $\mathcal{Z}$ is quasi-Gr\"obner. The image of the functor $\mathcal{Z}\to\mathbf{FS}^{op}$ lies in $\mathbf{OS}^{op}$, however this functor is not faithful, e.g.\@ there are two different morphisms from a graph with one vertex of degree $1$ and one loop to a graph with one vertex of degree $0$ with two loops, and these morphisms are mapped to the same map in $\mathbf{FS}^{op}$. We will construct an essentially surjective functor $\mathcal{G}\to\mathcal{Z}$ that satisfies property (F) and is such that the composition $\mathcal{G}\to\mathcal{Z}\to\mathbf{OS}^{op}$ is faithful, which implies that $\mathcal{G}$ satisfies property (G1). To prove that $\mathcal{Z}$ is quasi-Gr\"obner it then suffices to prove that $\mathcal{G}$ satisfies property (G2). 

An object of $\mathcal{G}$ is an object $p$ of $\mathcal{Z}$ additionally endowed with a colour map $col:H\to\mathbb{N}$ from the half-edges of $p$ that satisfies the following: 
\begin{enumerate}
    \item For any edge of $p$ its two  half-edges have the same colour, which will be called the colour of the edge.
    \item Let $p'$ be the operadic graph obtained from $p$ by repeated removal of all the vertices of degree $1$ and genus $0$ together with the adjacent edges (i.e.\@ we remove both of the half-edges of the only edge adjacent to the corresponding vertex), until no vertex of arity $0$ and genus $0$ remains, except possibly the last vertex. Then the removed edges have colour $0$ and the remaining edges have colour different from $0$.
    \item If a path in $p'$ from a vertex $v$ to a vertex $w$ (here the path includes the vertices $v$ and $w$) consists only of vertices of degree $2$ and genus $0$, then the edges in this path have the same colour. Let $p''$ be the graph obtained from $p'$ by replacing each such maximal path with an edge that has the same colour as the edges in the path that it replaces. Then any two half-edges of $p''$ that do not belong to the same edge have different colours.
    \item \label{item-condition-on-colours} For all $h$ in $H$ we have $col(h)\leq 9(\sum_{v\in p} g(v) + g(p) + l(p))$, where $g(v)$ is the genus of a vertex $v$, $g(p)$ is the genus of the graph of $p$, and $l(p)$ is the number of leaves of $p$.
\end{enumerate}
The morphisms are required to preserve colours, i.e.\@ for a morphism $f:p\to q$ that inserts $q_l$ into the $l$-th vertex of $p$ the leaves of $q_l$ (seen as the half-edges of $q$) have the same colours as the half-edges of $p$ to which these leaves correspond under $f$. 

If there is a morphism $f:p\to q$ in $\mathcal{Z}$, then the value of the right hand side of the inequality in the condition~(\ref{item-condition-on-colours}) is the same for $p$ and $q$. This implies that the forgetful functor $\mathcal{G}\to\mathcal{Z}$ is a discrete fibration, i.e.\@ that for any morphism $f:p\to q$ and colouring of $q$ there is exactly one colouring of $p$ compatible with $f$. The functor $\mathcal{G}\to\mathcal{Z}$ is surjective on objects. Indeed, the number of colours needed to colour an object $p$ in $\mathcal{Z}$ is the same as that of the graphs $p'$ and $p''$ from the above conditions, and is less or equal to $3(e + l)$, where $e$ and $l$ are the number of edges and of half-edges in the topological realization of $p'$ (or equally of $p''$). If $p'$ is the graph with one vertex of genus $0$ and degree $1$, then the condition~(\ref{item-condition-on-colours}) trivially allows to colour $p'$. Otherwise there is a morphism $f:p'\to q'$ in $\mathcal{Z}$ such that all vertices in $q'$ have genus $0$ and degree either $2$ or $3$. Add a vertex with a loop to each leaf of $q'$ to get a graph $r$ with 3-valent topological realization. The total genus of $r$ is $g_r=(\sum_{v\in p} g(v) + g(p) + l(p))$, and the number of edges in the topological realization of $r$ is $3g_r-3$. A colouring of $r$ gives a colouring of its subgraph $q'$, and in turn gives a colouring of $p'$ and of $p$.

The functor $\mathcal{G}\to\mathcal{Z}$ is a discrete fibration with finite non-empty fibers. Any discrete fibration $\mathcal{A}\to\mathcal{B}$ with finite non-empty fibers satisfies property (F): for an object $b$ the objects and the morphisms that ensure the property (F) for $b$ are all the objects in $\mathcal{A}$ over $b$ together with the morphism $id_b$.

The composition $\mathcal{G}\to\mathcal{Z}\to\mathbf{OS}^{op}$ is faithful: for a morphism $f:p\to q$ in $\mathcal{G}$ the subgraphs $q_l$ of $q$ that are inserted by $f$ into $p$ can be recovered as follows. Their vertices are given by the image of $f$ in $\mathbf{OS}^{op}$, and their edges of non-zero colour are the edges that have the colours of $q$ that are not the colours of $p$. By Lemma~\ref{lem:faithful-g1} the faithfulness of $\mathcal{G}\to\mathbf{OS}^{op}$ implies that $\mathcal{G}$ satisfies property (G1). 

It remains to prove property (G2) for $\mathcal{G}$. Let $f_i$ be a sequence of morphisms from $p$ in $\mathcal{G}$, with $f_i$ inserting $q_{li}$ into the $l$-th vertex of $p$. Let $q'_{li}$ be the coloured operadic graphs obtained from the graphs $q_{li}$ by repeated removal of all the vertices of degree $1$ and genus $0$ until no such vertices remain (except possibly the last vertex), and by further replacing the maximal paths whose vertices have degree $2$ with edges, and let $f'_i$ be the corresponding morphisms from $p$. Since the genus of the $l$-th vertex of $p$ is equal to $(\sum_{v\in q_{li}}g(v)+g(q_{li}))$ and the number of half-edges of the $l$-th vertex of $p$ is equal to the number of leaves of $q_{li}$, the total number of possible coloured operadic graphs $q'_{li}$ is finite, and  there is a subsequence $f'_j$ of the sequence $f'_i$ that consists of the same morphism $f:p\to q'$ repeated infinitely often. Let $f_j$ be the corresponding subsequence of $f_i$. Observe that there are morphisms $g_j$ such that $f_j=g_j\circ f$, where $g_j$ inserts planar operadic trees with two leaves into the vertices of $q'$. If $g_{j_1}\leq g_{j_2}$ then $f_{j_1}\leq f_{j_2}$. Let $r$ be the object of the subcategory $\mathcal{D}'$ of $\mathcal{C}(nusOp)^{op}$ (described in Corollary~\ref{thm:sop-grobner}) that is obtained from $q'$ by cutting all the edges that are not traversed via the depth-first search of $q'$, i.e.\@ by replacing these edges with two corresponding leaves. To morphisms $g_j$ correspond to morphisms $g'_j$ from $r$ in $\mathcal{D}$ such that $g_j$ and $g'_j$ insert the same planar trees in their vertices. Since $\mathcal{D}$ satisfies property (G2), there are some $j_1< j_2$ such that $g'_{j_1}\leq g'_{j_2}$, which implies that $g_{j_1}\leq g_{j_2}$, and $f_{j_1}\leq f_{j_2}$.
\end{proof}

\section{The general constructions and cobordisms}
\label{sec:general-tw}

The categories $\mathcal{C}(P)$ are closely related to more general constructions, described respectively in \cite{ginzburg1994koszul} and \cite{burkin2022twisted}. We recall them briefly and show that these constructions allow to find new quasi-Gr\"obner categories. 

In the following definitions, given an algebra $A$ over an operad $Q$, the expression $q(a_i,\dots,a_{i+m-1})$ denotes the element of $A$ obtained by the application of an operation $q$ to elements $a_i,\dots,a_{i+m-1}$ of $A$.

\begin{definition}
Let $Q$ be a $C$-coloured operad and $A$ be a $Q$-algebra. The universal enveloping category $\mathcal{U}_Q(A)$ of $A$ has the colours of $Q$ as its object. A morphism in $\mathcal{U}_Q(A)$ is an equivalence class of expressions $p(c_1,a_2,\dots,a_n)$ where $p\in Q(c_1,\dots,c_n;c_0)$ for some $c_i$ in $C$ and $a_i$ in $A(c_i)$. The equivalence is generated by relations  $(p\circ_i q)(c_1,a_2,\dots,a_{n+m-1})\sim p(c_1,a_2,\dots,a_{i-1},q(a_i,\dots),a_{i+m},\dots)$, where $2\leq i\leq n$ and $q$ is composable with $p$, and by relations $p^\sigma(c_1,a_{\sigma(2)},\dots,a_{\sigma(n)})\sim p(c_1,a_2,\dots,a_n)$ for all permutations $\sigma\in \mathbb{S}_n$ that preserve $1$. The source of a morphism $p(c_1,a_2,\dots,a_n)$ is $c_1$, and the target is $c_0$, where $p$ is in $Q(c_1,\dots,c_n;c_0)$. 
\end{definition}

\begin{definition}
Let $Q$ be a $C$-coloured operad and $A$ be a $Q$-algebra. The twisted arrow category $\Tw_Q(A)$ of $A$  has the elements of $A$ as its objects. A morphism in $\Tw_Q(A)$ is an equivalence class of expressions $p(a_1,a_2,\dots,a_n)$, where $p\in Q(c_1,\dots,c_n;c_0)$ for some $c_i$ in $C$ and $a_i$ in $A(c_i)$. The equivalence is generated by relations $(p\circ_i q)(a_1,a_2,\dots,a_{n+m-1})\sim p(a_1,a_2,\dots,a_{i-1},q(a_i,\dots),a_{i+m},\dots)$, where $2\leq i\leq n$ and $q$ is composable with $p$, and by relations $p^\sigma(a_1,a_{\sigma(2)},\dots,a_{\sigma(n)})\sim p(a_1,a_2,\dots,a_n)$ for all permutations $\sigma\in \mathbb{S}_n$ that preserve $1$. The source of a morphism $p(a_1,a_2,\dots,a_n)$ is $a_1$, and the target is $p(a_1,a_2,\dots,a_n)$ seen as an element of $A(c_0)$. 
\end{definition}

In both cases the composition of morphisms is computed via the operadic substitution $\circ_1$. The functor $G:\Tw_Q(A)\to\mathcal{U}_Q(A)$ that sends an element of $A$ to the corresponding colour of $Q$ is a discrete opfibration. Notice that $\mathcal{U}_Q(A)$-modules are the same as $A$-modules. If $A$ is the terminal $Q$-algebra then $\Tw_Q(A)=\mathcal{U}_Q(A)$. 

\subsection*{Cobordisms} An example of the construction $\Tw$ is the category $\mathbf{CS}$ whose objects are connected orientable surfaces with indexed boundaries and morphisms are cobordisms between boundaries. The subcategory $\mathbf{ncCS}$ of $\mathbf{CS}$ may serve instead of $\mathbf{FS}^{op}$ as the indexing category of modules related to modular operads, e.g.\@ of the module $H_i(\overline{M}_{g,n})$ considered in \cite{tosteson2019representation}. 

\begin{proposition}
Let $mOp$ be the operad whose algebras are modular operads, and $uCom_m$ be the modular envelope of the terminal cyclic operad. Let $Cob$ be the category with finite sequences of circles as objects and orientable cobordisms between disjoint unions of the indexed circles as morphisms. The category $\mathcal{U}_{mOp}(uCom_m)$ is isomorphic to the subcategory $Cob'$ of $Cob$ that contains all the morphisms of $Cob$ except the non-trivial morphisms from the empty sequence $\emptyset$. The category $\Tw_{mOp}(uCom_m)$ is isomorphic to the category $\mathbf{CS}$, the full subcategory of the slice category $\emptyset/Cob$ on non-empty connected cobordisms from the empty set. The projection functor $G:\mathbf{CS}\to Cob'$ is a discrete opfibration surjective on objects.
\end{proposition}
\begin{proof}
For the definition of the operad $mOp$ and for the proof see Section~1.1 and Proposition~2.33 in \cite{burkin2022twisted}. 
\end{proof}

\begin{proposition}
The category $\mathbf{CS}$, and thus the category $Cob'$, are not quasi-Gr\"obner.
\end{proposition}
\begin{proof}
Let $M$ be the module over $\mathbf{CS}$ equal to $\mathbb{Q}$ on all objects, with maps $\mathbb{Q}\to\mathbb{Q}$ being the identity maps. This module is finitely generated (by the hemisphere). The submodule $N$ of $M$ that is equal to $0$ on the surfaces with non-trivial boundary and to $\mathbb{Q}$ on the closed surfaces is not finitely generated.
\end{proof}

The proof shows that a quasi-Gr\"obner subcategory of $\mathbf{CS}$ cannot encode the module structure maps $M_{g,n}\to M_{g+k,n-2}$ for all $k$.

\begin{definition}
The categories $ncCob'$ and $\mathbf{ncCS}$ are the wide subcategories of $Cob'$ and $\mathbf{CS}$ on all morphisms such that the target boundary of each connected component of the corresponding cobordism is non-trivial. 
\end{definition}

To prove that the categories $ncCob'$ and $\mathbf{ncCS}$ are quasi-Gr\"obner we use the following category.

\begin{definition}
The category $\mathbf{gOS}$ of graded ordered surjections has the sets $\underline{n}=\{1,\dots,n\}$ as objects for all $n\geq 0$. A morphism $f:\underline{n}\to\underline{m}$ in $\mathbf{gOS}$ is a surjective map $f:\underline{n}\to\underline{m}$ such that $\min f^{-1}(i)<\min f^{-1}(j)$ for all $i<j$ in $\underline{m}$, together with a map $g_f:\underline{m}\to\mathbb{N}$ called grading. Composition $h\circ f$ of $f:\underline{n}\to\underline{m}$ and $h:\underline{m}\to\underline{k}$ is given by the composition of the corresponding set-maps and by the grading $g_{h\circ f}(i)=g_h(i)+\sum_{j\in h^{-1}(i)} g_f(j)$.
\end{definition}

\begin{definition}
    The functor $\Phi:\mathbf{gOS}^{op}\to ncCob'$ sends $\underline{n}$ to the sequence of $n$ circles and sends the opposite of $f:\underline{n}\to\underline{m}$ to the cobordism that for all $i$ connects the $i$-th circle with the circles in $f^{-1}(i)$ by the surface of genus $g_f(i)$.
\end{definition}

Notice that the functor $\Phi$ is faithful and surjective on objects, thus we can consider $\mathbf{gOS}^{op}$ as a wide subcategory of $ncCob'$.

\begin{proposition}
The category $\mathbf{gOS}^{op}$ is Gr\"obner.
\end{proposition}
\begin{proof}
Denote by $\mathbf{OS}$ the category of ordered surjections from \cite{sam2017grobner}. Recall that the category $\mathbf{OS}^{op}$ is Gr\"obner. Let $U:\mathbf{gOS}^{op}\to\mathbf{OS}^{op}$ be the forgetful functor and denote by $\prec$ be the admissible orders on $\mathbf{OS}^{op}$. Define the orders $\prec'$ on the morphisms of $\mathbf{gOS}^{op}$ from the same targets so that $f\prec h$ if $U(f)\prec U(h)$ or if $U(f)=U(h)$ and for some $i$ we have $g_f(j)=g_h(j)$ for all $j<i$, and $g_f(i)<g_h(i)$. This is an admissible order, and property (G1) holds. 

To show that property (G2) holds take a sequence of morphisms $f_1,f_2,\dots$ in $\mathbf{gOS}^{op}$ with the same source. Since (G2) holds for $\mathbf{OS}^{op}$ we can take an infinite subsequence $f_{i_1},f_{i_2},\dots$ of this sequence so that $U(f_{i_1}),U(f_{i_2}),\dots$ is non-decreasing. Then we can further take subsequences of $f_{i_1},f_{i_2},\dots$ so that the value of the grading on the first element of the source is non-decreasing, then the value of the grading on the second element is non-decreasing, and so on until the value on the last element of the source is non-decreasing. This gives an infinite non-decreasing subsequence of $f_1,f_2,\dots$ in the order $\leq$, i.e.\@ property (G2) holds.
\end{proof}

\begin{proposition}
The categories $\mathbf{ncCS}$ and $ncCob'$ are quasi-Gr\"obner.
\end{proposition}
\begin{proof}
The functor $\mathbf{ncCS}\to ncCob'$ is a discrete opfibration. By Lemma \ref{lem:opfibration-quasi-grobner} if $ncCob'$ is quasi-Gr\"obner, then $\mathbf{ncCS}$ is quasi-Gr\"obner. 

Let $\Phi:\mathbf{gOS}^{op}\to ncCob'$ be the inclusion described above. Let $x$ be an object in $ncCob'$ and let $f_i$ be all the morphisms in $ncCob'$ from $x$ such that for each connected component of the cobordism $f_i$ its target boundary is a circle and its genus is $0$. Any morphism from $x$ factors uniquely as a composition of some $f_i$ with a morphism in the image of $\Phi$. The morphisms $f_i$ ensure that $\Phi$ satisfies property~(F).
\end{proof}

\section{Further examples}
\label{sec:further-examples}

The most interesting instances of the previously described general constructions $\Tw$ and $\mathcal{U}$ are those that come from operads seen as algebras. Recall that any $C$-colored operad $P$ can be seen as an algebra over a certain operad $sOp_C$ whose algebras are $C$-coloured operads, described in \cite[1.5.6]{berger2007resolution}. 

\begin{definition}
Let $P$ be an operad. The twisted arrow category $\Tw(P)=\Tw_{sOp}(P)$ of $P$ (\cite{burkin2022twisted,hoang2020quillen})  has the operations of $P$ as its objects. Morphisms are represented by planar rooted trees with half-edges. A morphism $f$ in $\Tw(P)$ from an operation $p$ of arity $n$ to an operation of arity $m$ corresponds to unique tree of height~3 (or height~2 if $n=0$), with $m$ input leaves, with exactly one middle vertex, which is marked by the source $p$ of $f$, with the remaining vertices marked by some operations $q_0$, $q_1,\dots, q_n$ from $P$, with the number of input edges of each vertex equal to the arity of the operation that marks this vertex, with half-edges coloured by the corresponding colours of operations so that two halves of the same edge have the same colour, with the lower vertex connected to the middle vertex by its first input edge, with leaves indexed from $1$ to $m$ so that for each vertex the indices of leaves above it increase in planar order. The target of a morphism is computed by the evaluation of the corresponding tree. Composition $f\circ g$ of trees $f$ and $g$ is computed by first grafting for all $i$ the $i$-th upper vertex of $f$ into the leaf of $g$ indexed by $i$,  grafting the root of $g$ to the first input edge of the lower vertex of $f$, which produces a tree of height 5, and then evaluating the maximal subtrees of this tree that do not contain the middle vertex. See Figure~\ref{fig:3leveltrees} for example.
\end{definition}

\begin{figure}[t]
\begin{align*}
  \begin{aligned}\begin{tikzpicture}[optree,
    level distance=18mm,
    level 2/.style={sibling distance=30mm},
    level 3/.style={sibling distance=20mm}]
  \node{}
  child { node[circ,label=-3:$q_0$]{}
    child { node[circ,label=183:$p$]{}
      child { node[circ,label=180:$q_1$]{}
        child { node[label=above:$3$]{} }
        child { node[label=above:$5$]{} } }
      child { node[circ,label=-3:$q_2$]{} }
      child { node[circ,label=-3:$q_3$]{}
        child { node[label=above:$1$]{} } } }
    child { node[label=above:$2$]{} }
    child { node[label=above:$4$]{}}};
\end{tikzpicture}\end{aligned}
\end{align*}
\caption{A tree that represents a morphism from operation $p$ of arity 3 to operation of arity 5 in the twisted arrow category $\Tw(P)$ of an operad $P$.}
\label{fig:3leveltrees}
\end{figure}
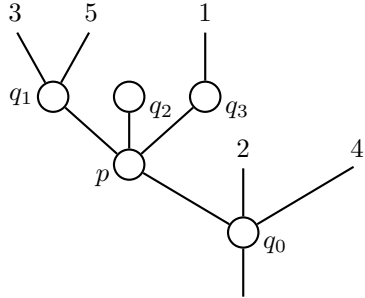

\begin{definition}
The universal enveloping category $\mathcal{U}(P)=\mathcal{U}_{sOp_C}(P)$ of a $C$-coloured operad $P$ has tuples $(c_0,\dots,c_n)$ of elements of $C$ as objects. Its morphisms correspond to the same trees as above, except the middle vertex is not marked by an operation, and in a morphism from $(c_0,\dots,c_n)$ for all $i$ the colour of the $i$-th edge adjacent to the middle vertex is $c_i$, where the $0$-th edge is the output edge.
\end{definition}

The category $\mathcal{U}(P)$ first appears in \cite{fresse2014functor} as the opposite of the category of pointed operators of $P$, i.e.\@ as the category $(\Gamma_P^+)^{op}$.

\begin{remark}
For any operad $P$ the category $\mathcal{C}(P)^{op}$ from \cite{deBrito2018catp} is the wide subcategory of $\Tw(P)$ on morphisms such that the lower vertex is marked by an identity operation. Similarly, the opposite of the PROP corresponding to an operad $P$ is the wide subcategory of $\mathcal{U}(P)$ on morphisms such that the lower vertex is marked by an identity operation. 
\end{remark}

These categories form the following commutative diagram.

\[\begin{tikzcd}
	{\mathcal{C}(P)^{op}} & {\Tw(P)} \\
	{\mathrm{PROP}(P)^{op}} & {\mathcal{U}(P)}
	\arrow[from=1-1, to=1-2, hook]
	\arrow[from=2-1, to=2-2, hook]
	\arrow[from=1-1, to=2-1, "G"']
	\arrow[from=1-2, to=2-2, "G"']
\end{tikzcd}\]
 
We make the following obvious observation.

\begin{lemma}
The functors $G:\mathcal{C}(P)^{op}\to \mathrm{PROP}(P)^{op}$ and $G:\Tw(P)\to\mathcal{U}(P)$ are discrete opfibrations.
\end{lemma}

\begin{corollary}
For any set-operad $P$ if $\mathcal{U}(P)$ is quasi-Gr\"obner then $\Tw(P)$ is quasi-Gr\"obner. If $\mathrm{PROP}(P)^{op}$ is quasi-Gr\"obner then $\mathcal{C}(P)^{op}$ is quasi-Gr\"obner. If $\Tw(P)\to\mathcal{U}(P)$ is surjective on objects, then the opposite implications hold.
\end{corollary}
\begin{proof} 
This follows from Lemma~\ref{lem:grobner-opfibration-left-right} and Lemma~\ref{lem:opfibration-quasi-grobner}.
\end{proof}

In general one should not expect the categories $\Tw(P)$ and $\mathcal{U}(P)$ to be quasi-Gr\"obner, even for reasonable operads $P$.

\begin{proposition}
\label{pr:dendroidal-not-grobner}
The twisted arrow category $\Tw(sOp)$ of the operad of single-coloured operads, equivalently the Moerdijk--Weiss category $\Omega$, is not quasi-Gr\"obner.
\end{proposition}
\begin{proof}
For $i\geq 3$ let operations $p_i$ be the trees with two vertices, with the root vertex having 2 input edges, with the second vertex attached to the first input edge of the root vertex, with the second vertex having $i$ input edges. Let $M$ be the $\mathbb{Q}$-module $\mathbb{Q}Hom(id_2,-)$ over $\Tw(sOp)$ generated by $id_{2}$, the tree with one vertex with two input leaves. Up to automorphisms, there is only one morphism $id_{2}\to p_i$ for all $i$, and there is no operation $q$ with morphisms $q\to p_i$ and $q\to p_j$ for $i\neq j$. Let $N$ be the submodule of $M$ such that $N(p)=M(p)$ if $p$ has at least 3 leaves and $N(p)=0$ otherwise. This module is not finitely generated: $N(p_i)\neq 0$ for all $i$, and these cannot be generated by a finite number of objects. Thus $\Omega$ is not quasi-Gr\"obner.
\end{proof}

Similar reasoning shows that the twisted arrow categories (and thus the universal enveloping categories) of the operads that encode planar, cyclic, modular operads, PROPs, wheeled PROPs, properads, and similar operad-like structures are not quasi-Gr\"obner. However, for more basic operads we recover some classical examples. Table~\ref{tab:my_label} lists the categories $\Tw(P)$, $\mathcal{U}(P)$, $\mathcal{C}(P)^{op}$ and $\mathrm{PROP}(P)^{op}$ (or in some cases the skeletons of these categories) for the operads $uAs$ of monoids, $As$ of semigroups, $uCom$ of commutative monoids, and $Com$ of commutative semigroups. 

\begin{table}[t]
    \centering
\begin{tabular}{|l|l|l|l|l|l|l|}
\hline
$P$ & $\Tw(P)$ & $R_{>0}$ & $\mathcal{C}(P)^{op}$ & $R_{>0}\cap \mathcal{C}(P)^{op}$ & $\mathcal{U}(P)$ & $\mathrm{PROP}(P)^{op}$\\
\hline
$uAs$ & $\Delta$ & $\mathbf{OI}_+$ & $\Delta_{ep}$ & $\mathbf{OI}_{+ep}$ & $\mathbf{FA}_*(\mathrm{as})^{op}$ & $\mathbf{FA}(\mathrm{as})^{op}$ \\
$As$ & $\mathbf{OI}_{++}$ & $\mathbf{OI}_{++}$ & $\mathbf{OI}_{++ep}$ & $\mathbf{OI}_{++ep}$ & $\mathbf{FS}_*(\mathrm{as})^{op}$ & $\mathbf{FS}(\mathrm{as})^{op}$ \\
$uCom$ & $\mathbf{FA}_*^{op}$ & $\mathbf{FS}_*^{op}$ & $\mathbf{FA}^{op}$ & $\mathbf{FS}^{op}$ & $\mathbf{FA}_*^{op}$ & $\mathbf{FA}^{op}$ \\
$Com$ & $\mathbf{FS}_{+*}^{op}$ & $\mathbf{FS}_{+*}^{op}$ & $\mathbf{FS}_+^{op}$ & $\mathbf{FS}_+^{op}$ & $\mathbf{FS}_*^{op}$ & $\mathbf{FS}^{op}$\\
\hline
\end{tabular}
    \caption{Classical examples.}
    \label{tab:my_label}
\end{table}

In this table the category $R_{>0}$ is the wide subcategory of $\Tw(P)$ that consists of morphisms such that all the vertices in the corresponding trees, except possibly the source vertex, have non-zero arity (the notation $R_{>0}$ comes from the generalized Reedy structure on the categories in the table). The category $\Delta$ is the simplex category; $\mathbf{OI}$ is the category of finite ordinals and order preserving injections; $\Delta_{ep}$ is the interval category, i.e.\@ the wide subcategory of $\Delta$ on endpoint-preserving maps; $\mathbf{FA}(\mathrm{as})$ is the category of non-commutative sets (the category $\Delta\mathbf{S}$ from \cite{fiedorowicz1991crossed}, see also \cite{pirashvili25hochschild}, except in our example the category $\mathbf{FA}(\mathrm{as})$ additionally contains the empty set); $\mathbf{FS}(\mathrm{as})$  is the subcategory of surjections in $\mathbf{FA}(\mathrm{as})$; $\mathbf{FA}$ is the category of finite sets; $\mathbf{FS}$ is the subcategory of surjections in $\mathbf{FA}$. The subscripts  denote the corresponding subcategories on sets that have $(+)$ at least one element, $(++)$ at least two elements, $(*)$ a marked element that is preserved by the maps, $(+*)$ at least one element in addition to the marked element, while $(ep)$ denotes the wide subcategory of endpoint-preserving maps. Notice that $\Delta_{ep}$ is equivalent to $\Delta^{op}$. 

\begin{lemma}
\label{lem:r-zero-to-full}
Let $P$ be an operad such that the set $P(;c)$ is finite for all colours $c$ of $P$, and let $\cA$ be the category $\Tw(P)$, $\mathcal{U}(P)$, $\mathcal{C}(P)^{op}$, or $\mathrm{PROP}(P)^{op}$. Let $R_{>0}$ be the wide subcategory of $\cA$ on morphisms represented by trees with all the non-source vertices marked by operations of non-zero arity. If $R_{>0}$ is quasi-Gr\"obner, then $\cA$ is quasi-Gr\"obner.
\end{lemma}
\begin{proof}
The inclusion $R_{>0}\to \cA$ is essentially surjective. For an object $p$ of $\cA$ let $f_i$ be all the morphisms from $p$ represented by trees whose non-source vertices are marked either by identity operations or by operations of arity 0, with leaves permuted trivially. These morphisms ensure the property (F) for the inclusion $R_{>0}\to \cA$.
\end{proof}

\begin{proposition}
The categories in Table~\ref{tab:my_label} are quasi-Gr\"obner.
\end{proposition} 
\begin{proof}
The only new examples are the categories related to non-commutative sets. We consider other examples for the sake of completeness. We rely on the fact that the categories $\mathbf{OI}$, $\mathbf{FA}^{op}$ and $\mathbf{FS}^{op}$ are  Gr\"obner or quasi-Gr\"obner.

Since the category $\mathbf{OI}$ is Gr\"obner, the categories $\mathbf{OI}_+$ and $\mathbf{OI}_{++}$ are Gr\"obner. Lemma~\ref{lem:r-zero-to-full} implies that $\Tw(uAs)\simeq\Delta$ is quasi-Gr\"obner. The inclusions $\mathbf{OI}_{+ep}\to\mathbf{OI}_+$ and $\mathbf{OI}_{++ep}\to\mathbf{OI}_{++}$ satisfy property (S), therefore $\mathbf{OI}_{+ep}$ and $\mathbf{OI}_{++ep}$ are Gr\"obner, and by Lemma~\ref{lem:r-zero-to-full} the category $\Delta_{ep}$ is quasi-Gr\"obner. The discrete opfibrations $\Tw(uAs)\to\mathcal{U}(uAs)$ and $\mathcal{C}(uAs)^{op}\to\mathrm{PROP}(uAs)^{op}$ are surjective on objects, which implies that $\mathbf{FA}_*(\mathrm{as})^{op}$ and $\mathbf{FA}(\mathrm{as})^{op}$ are quasi-Gr\"obner. There is a discrete opfibration from the subcategory $R_{>0}$ of $\Tw(uAs)$ to $\mathcal{U}(As)$ that is surjective on objects, and its restriction $\mathcal{C}(uAs)^{op}\cap R_{>0}\to \mathrm{PROP}(As)^{op}$ is also a discrete opfibration surjective on objects, which implies that $\mathbf{FS}_*(\mathrm{as})^{op}$ and $\mathbf{FS}(\mathrm{as})^{op}$ are quasi-Gr\"obner.

Since $\mathbf{FS}^{op}$ is quasi-Gr\"obner, $\mathbf{FS}_+^{op}$ is quasi-Gr\"obner. Recall that the category $\mathbf{OS}^{op}$, the opposite of the category of ordered finite surjections, is Gr\"obner. Consider its objects as the sets $\{0,\dots,n\}$, with $0$ as the marked element that is preserved by the maps. The inclusion $\mathbf{OS}^{op}\to \mathbf{FS}_*^{op}$ is surjective on objects and satisfies property (F) (with the morphisms $f_i$ being all automorphisms of an object), thus $\mathbf{FS}_*^{op}$ is quasi-Gr\"obner, and thus $\mathbf{FS}_{+*}^{op}$ is quasi-Gr\"obner. Lemma~\ref{lem:r-zero-to-full} implies that $\mathbf{FA}_*^{op}$ is quasi-Gr\"obner. 
\end{proof}

This leads to another example of a quasi-Gr\"obner category.

\begin{proposition}
Connes cyclic category $\Lambda$ and its subcategory of cyclic injections are quasi-Gr\"obner.
\end{proposition}
\begin{proof}
There is discrete opfibration $\Lambda\simeq\Tw_{cOp}(uAs_c)\to\mathcal{U}_{cOp}(uAs_c)\simeq\mathbf{FA}(\mathrm{as})^{op}$, where $cOp$ is the operad whose algebras are cyclic operads, and $uAs_c$ is the cyclic operad of monoids, see \cite[Proposition~2.31 and 2.32]{burkin2022twisted}. Its restriction to the category of injections (or to the opposite of the category of surjections, which is the same category by the self-duality $\Lambda\simeq\Lambda^{op}$) is a discrete opfibration over $\mathbf{FS}(\mathrm{as})^{op}$.
\end{proof}

The remaining examples come from operadic categories.

\begin{definition}
Let $S$ be a semigroup. Let $N(S)$ be the presheaf over $\mathbf{OI}_{++}$ such that $N(S)([n])=S^n$, with presheaf maps defined in the same way as in the nerve construction for monoids. Take the restriction of $N(S)$ to $\mathbf{OI}_{++ep}$. The category $\mathbf{OI}_{++ep}/N(S)$ has finite non-empty sequences of elements $s_i$ of $S$ as objects. Morphisms correspond to substitutions of elements $s_i$ by sequences $s_{i1},\dots,s_{ik_i}$ such that $s_{i1}\cdots s_{ik_i}=s_i$. 
\end{definition}

The category $\mathbf{OI}_{++ep}/N(S)$ is the opposite of the operadic category described in \cite[A.1]{mozgovoy2022operadic}. In the definition of the latter category the subcategory of surjections of $\Delta$ is used instead of $\mathbf{OI}_{++ep}^{op}$, but these two categories are equivalent. Notice that the categories $\mathbf{OI}_{++}/N(S)$ and $\mathbf{OI}_{++ep}/N(S)$ can be seen as the categories $\Tw(P)$ and $\mathcal{C}(P)^{op}$ where $P$ is the Baez--Dolan plus construction of the semigroup $S$ seen as an algebra over the operad $As$ of semigroups. 

\begin{proposition}
Let $S$ be a semigroup. If $S$ is not finite, then $\mathbf{OI}_{++}/N(S)$ is not quasi-Gr\"obner.
\end{proposition}
\begin{proof}
The proof is analogous to that of Proposition~\ref{pr:dendroidal-not-grobner}: the role of operations $p_i$ is played by sequences $(s_1,s_i)$ for pairwise different $s_i$ in $S$, and the role of morphisms $id_2\to p_i$ is played by inclusions $(s_1)\to (s_1,s_i)$.
\end{proof}

\begin{proposition}
Let $S$ be a semigroup such that for any element $s$ in $S$ the number of decompositions of $s$ into a product of non-identity elements of $S$ is finite. Then the category $\mathbf{OI}_{++ep}/N(S)$ is Gr\"obner.
\end{proposition}
\begin{proof}
The projection $\mathbf{OI}_{++ep}/N(S)\to \mathbf{OI}_{++ep}$ is faithful. By Lemma~\ref{lem:faithful-g1} the category $\mathbf{OI}_{++ep}/N(S)$ satisfies property (G1).

If $S$ does not have the identity element, then the slice categories of $\mathbf{OI}_{++ep}/N(S)$ are finite, and the property (G2) holds. Assume that $S$ has the identity element $e$. The condition on $S$ implies that $e$ cannot be decomposed into non-trivial product of elements of $S$.  To prove property (G2) let $x=(s_1,\dots,s_n)$ be an object in $\mathbf{OI}_{++ep}/N(S)$ and let $f_i$, $i\in \mathbb{N}$, be a sequence of morphisms from $x$. These morphisms correspond to decompositions of elements $s_j$ into products of elements of $S$. Since the number of these decompositions, up to multiplication by $e$, is finite, there is a subsequence of $f_i$ such that for each $j$ the morphisms decompose the element $s_j$ in the same way up to multiplication by $e$. We can further choose a subsequence such that for each $j$ the number of elements $e$ between any two adjacent non-trivial elements in the decomposition of $s_j$, and also the number of elements $e$ before the first non-trivial element and the number of elements $e$ after the last non-trivial element, is non-decreasing. This gives a non-decreasing subsequence of morphisms.
\end{proof}

\begin{proposition}
Let $S$ be the group $\mathbb{Z}/2$. The category $\mathbf{OI}_{++ep}/N(S)$ is Gr\"obner.
\end{proposition}
\begin{proof}
Again the projection $\mathbf{OI}_{++ep}/N(S)\to \mathbf{OI}_{++ep}$ is faithful, which implies property (G1).

Let $x=(s_1,\dots,s_n)$ be an object in $\mathbf{OI}_{++ep}/N(S)$ and let $f_i$, $i\in \mathbb{N}$, be a sequence of morphisms from $x$. We can choose a subsequence $f_i$ such that for all $j$ the number of symbols $1$ in the subsequence $(s'_{i_j},\dots,s'_{i_{j+1}-1})$ that replaces $s_j$ is either stable or increases, and if it is stable, we can further choose subsequence such that the numbers of symbols $0$ to the left of the first, in between of $l$-th and $(l+1)$-th for all $l$, and to the right of the last symbol $1$ are non-decreasing. Take any $f_i$ in the subsequence. Let $h$ be the morphism from the target of $f_i$ that, for all $j$ such that the number of symbols $1$ in $(s'_{i_j},\dots,s'_{i_{j+1}-1})$ increases with $i$, replaces the symbols $0$ in subsequence $(s'_{i_j},\dots,s'_{i_{j+1}-1})$ by subsequence $(1,1)$. There is some $N$ such that there is a morphism from $h\circ f_j$ to morphism $f_l$ in the subsequence for all $l>N$, i.e.\@ $f_j\leq f_l$.
\end{proof}

\begin{example}
Consider the opposite of the category $\mathcal{C}$ from \cite[Example~3]{garner2021operadic}. This category is obtained from the commutative monoid $\mathbb{R}_{\geq 0}$ as follows. Let $N(\mathbb{R}_{\geq 0})$ be the nerve of $\mathbb{R}_{\geq 0}$, seen as a presheaf over the Segal's category $\mathbf{FA}_*^{op}$, and consider the restriction of $N(\mathbb{R}_{\geq 0})$ to the subcategory $\mathbf{FA}^{op}$ of active morphisms. Then $\mathcal{C}^{op}$ is the full subcategory of $\mathbf{FA}^{op}/N(\mathbb{R}_{\geq 0})$ on objects $(s_1,\dots,s_n)$ such that $\sum_i s_i\leq 1$. 

Let $M$ be the module generated by $(1)$ and $N$ be the submodule of $M$ such that $N((s_1,\dots,s_n))=0$ if all $s_i$ except one are equal to $0$, and $N(x)=M(x)$ on the remaining objects $x$. There are morphisms $(1)\to (p, 1-p)$ for all $p\in(0,1)$, thus $N((p,1-p))\neq 0$. The module $N$ is not finitely generated, and the category $\mathcal{C}^{op}$ is not quasi-Gr\"obner.
\end{example}

\bibliographystyle{amsalpha}
\bibliography{main}

\providecommand{\bysame}{\leavevmode\hbox to3em{\hrulefill}\thinspace}
\providecommand{\MR}{\relax\ifhmode\unskip\space\fi MR }
\providecommand{\MRhref}[2]{%
  \href{http://www.ams.org/mathscinet-getitem?mr=#1}{#2}
}
\providecommand{\href}[2]{#2}
\begin{thebibliography}{BdBW18}

\bibitem[Bar]{barter2015noetherianity}
Daniel Barter, \emph{Noetherianity and rooted trees}, arXiv:1509.04228.

\bibitem[BdBW18]{deBrito2018catp}
Pedro Boavida~de Brito and Michael Weiss, \emph{Spaces of smooth embeddings and
  configuration categories}, J. Topol. \textbf{11} (2018), no.~1, 65--143.

\bibitem[BM]{batanin2018operadic}
Michael Batanin and Martin Markl, \emph{Operadic categories as a natural
  environment for {K}oszul duality}, arXiv:1812.02935.

\bibitem[BM07]{berger2007resolution}
Clemens Berger and Ieke Moerdijk, \emph{Resolution of coloured operads and
  rectification of homotopy algebras}, Categories in algebra, geometry and
  mathematical physics, Contemp. Math., vol. 431, Amer. Math. Soc., Providence,
  RI, 2007, pp.~31--58.

\bibitem[BM15]{Batanin2015operadic}
Michael Batanin and Martin Markl, \emph{Operadic categories and duoidal
  {D}eligne's conjecture}, Adv. Math. \textbf{285} (2015), 1630--1687.

\bibitem[Bur22]{burkin2022twisted}
Sergei Burkin, \emph{Twisted arrow categories, operads and {S}egal conditions},
  Theory Appl. Categ. \textbf{38} (2022), Paper No. 16, 595--660.

\bibitem[FL91]{fiedorowicz1991crossed}
Zbigniew Fiedorowicz and Jean-Louis Loday, \emph{Crossed simplicial groups and
  their associated homology}, Trans. Amer. Math. Soc. \textbf{326} (1991),
  no.~1, 57--87.

\bibitem[Fre14]{fresse2014functor}
Benoit Fresse, \emph{Functor homology and operadic homology}, 2014.

\bibitem[GK94]{ginzburg1994koszul}
Victor Ginzburg and Mikhail Kapranov, \emph{Koszul duality for operads}, Duke
  Math. J. \textbf{76} (1994), no.~1, 203--272.

\bibitem[GKW21]{garner2021operadic}
Richard Garner, Joachim Kock, and Mark Weber, \emph{Operadic categories and
  d\'{e}calage}, Adv. Math. \textbf{377} (2021), Paper No. 107440, 23.

\bibitem[Hoa]{hoang2020quillen}
Truong Hoang, \emph{Quillen cohomology of enriched operads}, arXiv:2005.01198.

\bibitem[HRY20]{hackney2020graphical}
Philip Hackney, Marcy Robertson, and Donald Yau, \emph{A graphical category for
  higher modular operads}, Adv. Math. \textbf{365} (2020), 107044, 61.

\bibitem[Moz22]{mozgovoy2022operadic}
Sergey Mozgovoy, \emph{Operadic approach to wall-crossing}, J. Algebra
  \textbf{596} (2022), 53--88.

\bibitem[MW07]{moerdijk2007dendroidal}
Ieke Moerdijk and Ittay Weiss, \emph{Dendroidal sets}, Algebr. Geom. Topol.
  \textbf{7} (2007), 1441--1470.

\bibitem[PR02]{pirashvili25hochschild}
Teimuraz {Pirashvili} and Birgit {Richter}, \emph{{Hochschild and cyclic
  homology via functor homology}}, {\(K\)-Theory} \textbf{25} (2002), no.~1,
  39--49.

\bibitem[SS17]{sam2017grobner}
Steven~V. Sam and Andrew Snowden, \emph{Gr\"{o}bner methods for representations
  of combinatorial categories}, J. Amer. Math. Soc. \textbf{30} (2017), no.~1,
  159--203.

\bibitem[Tos19]{tosteson2019representation}
Philip Tosteson, \emph{Representation {S}tability, {C}onfigurations {S}paces,
  and {D}eligne-{M}umford {C}ompactifications}, ProQuest LLC, Ann Arbor, MI,
  2019, Thesis (Ph.D.)--University of Michigan.

\end{thebibliography}

\end{document}